\documentclass[a4paper,12pt]{article}
\usepackage[cp1251]{inputenc}
\usepackage{amssymb,amsmath,amsfonts}
\usepackage{graphicx}
\usepackage{caption}
\usepackage{hyperref}

\newtheorem{lemma}{Lemma}
\newtheorem{theorem}{Theorem}

\newtheorem{corollary}{Corollary}
\newtheorem{proposition}{Proposition}

\newenvironment{proof}{\noindent {\it Proof. }}{\hfill$\square$ \medskip}

\begin{document}
\renewcommand{\refname}{References}
\renewcommand{\figurename}{Fig.}
\renewcommand{\tablename}{Table}

\begin{center}
{\bf  \Large Coalition graphs of connected domination \\[1mm] partitions  in subcubic graphs}

\

{\large Andrey A. Dobrynin and Aleksey N. Glebov}

\vspace{2mm}

\textit{Sobolev Institute of Mathematics, Siberian Branch of the \\
Russian Academy of Sciences, Novosibirsk, 630090, Russia}

{\rm dobr@math.nsc.ru, angle@math.nsc.ru}

\end{center}

\

\textbf{\large Abstract}

\vspace{2mm}

\noindent
A graph is subcubic if it is connected and its maximum vertex degree does not exceed 3.
Two disjoint vertex subsets of a graph $G$  form a connected coalition in $G$
if neither of them is a connected dominating set but their union is
a connected dominating set. 
A connected coalition partition of $G$ is a partition of its vertices $\pi(G) = \{V_1, V_2,..., V_k \}$, 
such that each  $V_i$ is either a connected dominating set consisting of a single vertex or 
forms a coalition with some set of $\pi(G)$.
The formation of connected coalitions is described by a coalition graph  whose 
vertices correspond to the sets of $\pi$, and two vertices are adjacent  
 if and only if the corresponding  sets form a coalition in  $G$.
We characterize  all coalition graphs of subcubic graphs.

\section{Introduction}

In this paper, we consider simple graphs  $G(V,E)$ 
with the vertex set $V(G)$ and the edge set $E(G)$.
The \emph{order} of a graph is the number of its vertices.
The set of vertices adjacent to a vertex $v$ is denoted by $N(v)$.
The closed neighbourhood of a vertex $v$ is $N[v]=N(v) \cup \{v\}$.
The \emph{degree} $\deg(v)$ of a vertex $v$ is the cardinality of $N(v)$.
The maximum vertex degree in a graph $G$ is denoted by $\Delta(G)$.
A vertex $v$  with $\deg(v)=n-1$  in a graph $G$ of order $n$ is called a \emph{full vertex}.
A graph $G$ is called \emph{subcubic} if $G$ is connected and $\Delta(G) \le 3$.
 A subgraph $G[S]$ induced by $S \subseteq  V$  
  is the subgraph with the vertex set $S$, where 
two vertices of $S$ are adjacent in $G[S]$ if and only if they are adjacent in $G$.
By $K_n$, $P_n$, $C_n$, $S_n$  we denote
the complete graph, the simple path, the simple cycle, and the star graph 
of order $n$, respectively.
By $K_{p,q}$ we denote the complete bipartite graph with the parts of size $p$ and $q$.
A subset $D \subseteq V(G)$ is a \emph{dominating set} if every vertex
of $V(G)\setminus D$ is adjacent to at least one vertex of $D$.  
A dominating set $D \subseteq V(G)$ is a \emph{connected dominating set} if 
$G[D]$ is connected.
The domination theory in graphs has found  important applications 
in facility location, analysis of transportation and communication networks, etc.
For a detailed information of domination theory, we refer the reader to books
 \cite{Hayn2020-2,Hayn2023c,Hayn2021,Hayn1998}.  

Two disjoint vertex subsets of a graph $G$  form a \emph{connected coalition} in $G$
if neither of them is a connected dominating set but their union is
a connected dominating set. 
A connected coalition partition of $G$ is a partition of its vertices $\pi(G) = \{V_1, V_2,..., V_k \}$, 
such that each  $V_i$ is either a connected dominating set consisting of a single vertex or 
forms a coalition with some set of $\pi(G)$.
The \emph{connected coalition number} $CC(G)$ of a graph $G$
 is the largest possible number of subsets in $\pi(G)$.
The concept of  coalitions in graphs has been introduced in  \cite{Hayn2020}.
Various properties of coalitions have been studied in  
\cite{Alik2023a,Bakh2023,Dobr2024,Hayn2021a,Hayn2023a,Hayn2023b,Hayn2023d}.
A \emph{coalition graph} is a means of describing the formation of coalitions  \cite{Hayn2020}.
Given a graph $G$ and its connected coalition partition $\pi$, the coalition graph $CCG(G,\pi)$ is a graph with the 
vertex set $\{ V_1, V_2,\ldots, V_k \}$, where $V_i$ and $V_j$ are adjacent if and only if they form a coalition in $G$.
One of the problems in the study of coalition partitions for graphs of various classes is the characterization of  coalition graphs that arise here.
Coalition graphs of paths, cycles and trees have been considered  in  \cite{Bakh2023,Dobr2024-2,Gleb2025,Hayn2023a}. 

In this paper, we describe all coalition graphs based on connected coalition partitions
in subcubic graphs.

\section{Main result}

Almost all coalition graphs defined by connected coalitions of subcubic graphs can be obtained as coalition graphs of 
 M\"{o}bius ladders.
 M\"{o}bius ladder $M_n$ is formed from an even $n$-cycle by adding edges connecting opposite pairs of vertices in the cycle.
 It can be also constructed by introducing a twist in a prism graph $Pr_n$ of order $n$. 
 M\"{o}bius ladders and prisms are vertex-transitive graphs.  
Two representations of $M_n$ and a prism $Pr_n$ are  depicted in Fig.~\ref{Fig1}.

\begin{figure}[ht] 
\centering
\includegraphics[width=0.9\linewidth]{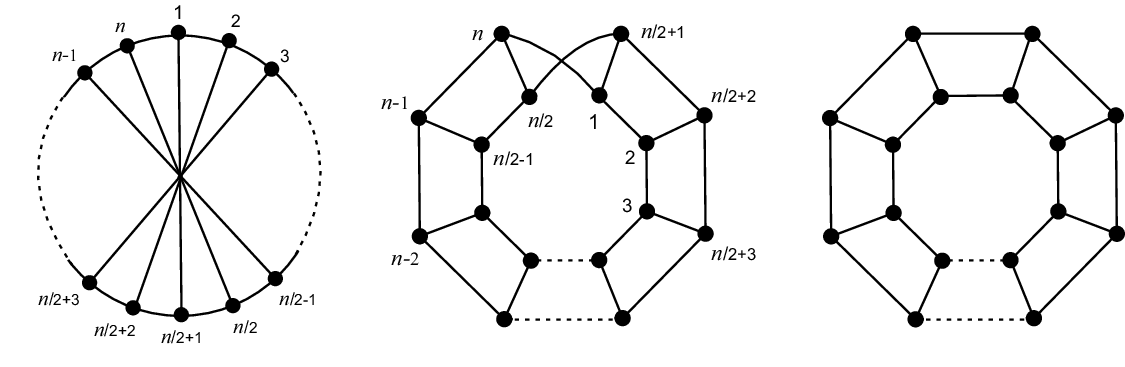}
\caption{M\"{o}bius ladder $M_n$ and prism $Pr_n$ of order $n$.}  \label{Fig1}
\end{figure}

The M\"{o}bius ladders $M_n$ of order $n \le 18$ define 21 coalition graphs of order at most 6.
The number of coalition graphs for $M_n$ are shown in Table~\ref{Tab1}.
These numerical results have been obtained by computer calculations.
Non-standard names of some graphs in the table are given in Fig.~\ref{Fig2}.


\begin{figure}[ht] 
\centering
\includegraphics[width=\linewidth]{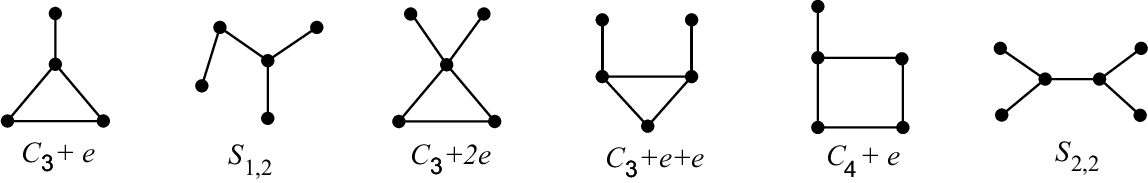}
\caption{Notation of coalition graphs from Table~\ref{Tab1}.}  \label{Fig2}
\end{figure}

\newpage

Based on these computational data, we formulate the following result. 

\begin{theorem} \label{Th1}
Subcubic graphs define the following coalition graphs:
the infinite set of stars $S_k$ for $k \ge 2$
and the finite set of 22 graphs of order at most 6:
$K_1$, $\overline{K}_2$, $\overline{K}_3$, $\overline{K}_4$, 
$C_3$, $2K_2$, $P_4$, $C_4$, $C_3+e$,
 $K_4-e$, $K_4$,  $P_2\cup P_3$, $S_{1,2}$, $P_5$,
$C_3+e+e$, $C_3+2e$, $C_4+e$, $C_5$, $S_{2,2}$,  $3K_2$,  $K_{2,3}$, and $K_{3,3}$.
\end{theorem}

To prove this statement, we present the corresponding coalition partitions
for the  graphs of Theorem~\ref{Th1} and 
show that there are no other coalition graphs defined by  subcubic graphs.

\begin{theorem} \label{Th2}
Coalition graphs 
$K_1$, $\overline{K}_2$,  $\overline{K}_3$, $K_4$, $\overline{K}_4$, $C_5$, $3K_2$, $K_{2,3}$, and $K_{3,3}$
are defined only by a finite number of subcubic graphs.
Other coalition graphs of Theorem~\ref{Th1} can be defined by an infinite number of subcubic graphs.
\end{theorem}

We will demonstrate that an infinite family of subcubic graphs of Theorem~\ref{Th2} 
is formed by  M\"{o}bius ladders and their modification. The proof is given in Sections 5 and 6.

\begin{table}[t!] 
\centering
\caption{Number of coalition graphs $CCG$  for M\"{o}bius ladder $M_n$, $6 \le n \le 18$.} \label{Tab1}
\begin{tabular}{rl@{\hspace{-2mm}}rrrrrrr}  \hline
 $N$  & $CCG$ &  $M_6$ &  $M_8$ &  $M_{10}$ &    $M_{12}$ &    $M_{14}$ & $M_{16}$ & $M_{18}$  \\  \hline
1. & $K_2$ &          1 &         21 &       176 &     1089 &         5538 &      25575 &        111844   \\  \hline  
 2. & $P_3$ &          6 &        238 &     3890 &   42282 &     361970 &  2711840 &   18731007      \\ 
 3. & $C_3$ &          .  &        184 &     1825 &   10830 &      54362 &     249280 &    1075418      \\   \hline 
 5. & $S_4$ &          2 &        108 &     2500 &     23380 &     172004 &  1139776 &   7104534      \\  
 4. & $2K_2$ &        . &          50 &        905 &     9816 &      70371 &     416016 &    2198475    \\ 
 6. & $P_4$ &          .  &        288 &     2370 &    12624 &      52430 &     190016 &     634734      \\ 
 7. & $C_3+e$ &      . &        184 &       970 &       3048 &       8456 &        21984 &       54972     \\ 
  8. & $C_4$ &          9 &         37 &        110 &        207 &          329 &          476 &           648      \\ 
 9. & $K_4-e$ &        . &         32 &          10 &          12 &            14 &             16 &            18      \\ 
10. & $K_4$ &          . &           6 &             . &             . &              . &                 . &                .      \\  \hline 
11. & $S_5$ &          . &           2 &         70 &          948 &       9492 &        82800 &    665856      \\ 
12. & $P_2 \cup P_3$ & . &    48 &       280 &         960 &       2604 &         6368 &       14796     \\ 
13. & $S_{1,2}$ &       . &         40 &       250 &          804 &       2114 &         5104 &       11790      \\ 
14. & $P_5$ &          . &         48 &       120 &          156 &        182 &            208 &          234      \\ 
15. & $C_3+2e$ &   . &           . &          10 &            24 &            56 &            128 &        288      \\ 
16. & $C_3+e+e$ & . &         16 &         20 &           24 &            28 &             32 &            36        \\ 
17. & $C_4 +e$ &    . &           8 &            . &             .  &               . &                  . &             .     \\ 
18. & $C_5$ &          . &          8 &            2 &              . &              . &                  . &             .     \\ 
19. &  $K_{2,3}$ &   6 &          . &             . &               . &              . &                  . &             .     \\  \hline 
20. & $S_{2,2}$ &       . &          4 &            5 &             6 &             7 &                 8 &            9      \\ 
21. & $K_{3,3}$   &   1 &          . &            . &               . &              . &                 . &              .      \\ 
\hline
\end{tabular}
\end{table}

\section{Lemma on connected dominating sets in subcubic graphs}

In this section, we state some useful results concerning  the structure of connected dominating sets 
in subcubic graphs.

\begin{lemma} \label{L1}
Let $D \subseteq V$ be a connected dominating set in $G$. Then 
$|D| \ge n/2 - 1$ 
and the induced subgraph $G[V \setminus D]$ has the maximum vertex degree at most 2. The equality 
$|D| = n/2 - 1$ 
holds if and only if all vertices in $D$ have degree 3, the induced subgraph $G[D]$ is a tree, and each vertex in $V \setminus D$ is adjacent to exactly one vertex of $D$.
\end{lemma}

\begin{proof} 
Since every vertex in $V \setminus D$ is adjacent to a vertex of $D$, the maximum vertex degree of the subgraph $G[V \setminus D]$ does not exceed 2. The subgraph $G[D]$ is connected, so it has at least $|D|-1$ edges and the sum of degrees of the vertices of $D$ inside $G[D]$ is at least $2(|D|-1)$. As $G$ is subcubic, the sum of degrees of the vertices of $D$ in $G$ does not exceed $3|D|$. 
Therefore, the number of edges joining vertices of $D$ with the vertices of $V \setminus D$ does not exceed $3|D|-2(|D|-1)=|D|+2$. Thus, $|V \setminus D| \le |D|+2$ and $n=|D|+|V \setminus D| \le 2|D|+2$ 
implying that 
$|D| \ge n/2 - 1$. 
The arguments above show that the equality holds if and only if each vertex of $D$ has degree 3 in $G$, each vertex of $V \setminus D$ is adjacent to exactly one vertex of $D$, and the number of edges in the subgraph $G[D]$ equals to $|D|-1$ which implies that $G[D]$ is a tree. 
\end{proof}
 
\begin{corollary} \label{Corol1}
Let $D_1,D_2 \subseteq V$ be two disjoint connected dominating sets in $G$. Then $|D_1|+|D_2| \ge n-2$ and each of the induced subgraphs $G[D_1]$ and $G[D_2]$ is a path or a cycle. If  $|D_1|+|D_2| = n-2$, then
$|D_1| = |D_2| =  n/2 - 1$, 
both subgraphs $G[D_1]$ and $G[D_2]$ are paths, and $D_1$ and $D_2$ are joined by a matching of size 
$n/2 - 1$ in $G$.
\end{corollary}

\begin{proof} 
Lemma \ref{L1} implies that 
$|D_1|+|D_2| \ge (n/2 - 1) + (n/2 - 1) = n-2$ 
and both subgraphs $G[D_1]$, $G[D_2]$ are connected and have the maximum vertex degree at most 2. Thus, each of these subgraphs is a path or a cycle. If  $|D_1|+|D_2| = n-2$, then 
$|D_1| = |D_2| = n/2 - 1$ 
and both subgraphs $G[D_1]$ and $G[D_2]$ are trees, so that they are paths. Finally, each vertex of $D_2$ is adjacent to exactly one vertex of $D_1$, and each vertex of $D_1$ is adjacent to exactly one vertex of $D_2$. This means that the vertex sets $D_1$ and $D_2$ are joined by a matching of size 
$n/2 - 1$ in $G$. 
\end{proof}

\begin{corollary} \label{Corol2}
Let $D_1,D_2,D_3 \subseteq V$ be three disjoint connected dominating sets in $G$.
If $n\ge 5$, then $|D_1| = |D_2| = |D_3| = 2$ and $G$ is isomorphic to one of the graphs $Pr_6$ or $M_6 \cong K_{3,3}$.
\end{corollary}

\begin{proof} 
Since $G$ is a subcubic graph of order $n \ge 5$, the size of a dominating set in $G$ is at least 2. Thus, $|D_i| \ge 2$ for $i=1,2,3$, and $n \ge |D_1|+|D_2|+|D_3| \ge 6$. 
By Lemma \ref{L1}, it follows that $n \ge |D_1|+|D_2|+|D_3| \ge 3(n/2 - 1)$, which implies that $n \le 6$. Thus, $n=6$ and $|D_1| = |D_2| = |D_3| = 2$. 
By Corollary \ref{Corol1},  the two vertices of $D_i$ are joined by an edge for $i=1,2,3$, and 
$D_i$ is joined with $D_j$ by a matching of size 2 in $G$
for any pair of indices $i,j \in \{1,2,3\}$. It is easy to check that $G$ is isomorphic to one of the graphs $Pr_6$ or $M_6$.
\end{proof}

\section{Properties of coalition graphs}

In this section, the structure and properties of coalition graphs are determined.
In the following lemmas we assume that a subcubic graph $G$ has a connected coalition partition 
$\pi(G) = \{V_1, V_2,..., V_k \}$ with the coalition graph $H=CCG(G,\pi)$. 
By $\alpha(H)$ we denote the maximal size of a matching in $H$.

\begin{lemma} \label{L2}
The coalition graph $H$ satisfies $\alpha(H) \le 3$. If $\alpha(H)=3$, then one of the following statements hold:

(i) $G \cong H \cong M_6 \cong K_{3,3}$, that is, $G$ is a self-coalition graph; 

(ii) $G \cong Pr_6$ and $H \cong 3K_2$. 
\end{lemma}

\begin{proof}
Suppose that $H$ contains a matching $(V_1,V_2)$, $(V_3,V_4)$, $(V_5,V_6)$. Then $D_1 = V_1 \cup V_2$, $D_2 = V_3 \cup V_4$, and $D_3 = V_5 \cup V_6$ are three disjoint connected dominating sets of size at least 2 in $G$. By Corollary \ref{Corol2}, we have $|D_1| = |D_2| = |D_3| = 2$,
and $G$ is isomorphic to one of the graphs $Pr_6$ or $M_6$. Hence $n=k=6$, $\alpha(H)=3$, and $V_i=\{v_i\}$ for $i=1,2,\ldots,6$. 
If $G \cong M_6$, then the set $V_i \cup V_j$ is a coalition for every edge $(v_i,v_j)$ of $G$. 
Thus, $H \cong G \cong M_6$. If $G \cong Pr_6$, then only $V_1 \cup V_2$, $V_3 \cup V_4$, and $V_5 \cup V_6$ are coalitions in $G$.
For any other pair of indices $i,j$, the set $V_i \cup V_j$ is either disconnected or is not dominating. Thus, $H \cong 3K_2$.
\end{proof}

\begin{lemma} \label{L3}
If $H$ contains an isolated vertex, then $n=k \le 4$ and $G \cong K_n$, $H \cong \overline{K}_n$.
\end{lemma}

\begin{proof} 
Let $V_i$ be an isolated vertex in $H$. 
Then $V_i = \{v_i\}$ for  some full vertex $v_i$ of $G$. So we have $n=d(v_i)+1 \le 4$. If $G \cong K_n$, then all vertices of $G$ are full and hence $H \cong \overline{K}_n$. Assume that $G$ is not complete. 
Since $G$ has a full vertex, it is isomorphic to one of the graphs $P_3$, $S_4$, $C_3+e$ or $K_4-e$. However, none of these graphs has a connected coalition partition because the set of all non-full vertices of $G$ is disconnected, hence it can not contain any connected coalition of $G$. 
\end{proof}

\begin{lemma} \label{L4}
If  $H$ contains a vertex $V_i$ of degree at least 4, then $V_i$ is a dominating set in $G$, and  the induced subgraph $G[V_i]$ is disconnected. 
\end{lemma}

\begin{proof} 
Suppose that $V_i$ is adjacent to vertices $V_1,V_2,V_3, V_4$ in $H$. Assume that $V_i$ is not a dominating set in $G$. Consider a vertex $v \in V$ such that $N[v] \cap V_i = \emptyset$. Since $V_i \cup V_j$ is a connected dominating set in $G$ for $j=1,2,3,4$, we have $N[v] \cap V_j \neq \emptyset$. Therefore, $d(v)=3$ and $N[v]$ contains exactly one vertex from each set $V_1,V_2,V_3,V_4$. 
Assume, without loss of generality, that $v \in V_1$ and $N(v) \subseteq V_2 \cup V_3 \cup V_4$. It follows that $v$ is an isolated vertex in the subgraph $G[V_1 \cup V_i]$.
This contradicts  the assumption that $V_1 \cup V_i$ is a connected dominating set in $G$. Hence $V_i$  is a dominating set in $G$ but the subgraph $G[V_i]$ is disconnected. 
\end{proof}

\begin{lemma} \label{L5}
If $\alpha(H)=1$, then $H$ is isomorphic to one of the graphs $K_3$ or $S_k$, where $k  \le \left\lfloor \frac {n+7}3 \right\rfloor$.
\end{lemma}

\begin{proof} 
By Lemma \ref{L3} and the equality $\alpha(\overline{K}_n)=0$, it follows that $H$ has no isolated vertices. As $\alpha(H)=1$, this implies that $H$ is connected. Let $V_i$ be a vertex of the maximum degree $d$ in $H$. Clearly, $d \ge 1$. If $d=1$, then $H \cong K_2 \cong S_2$. Suppose that $d \ge 2$ and let $V_1,V_2,\ldots,V_d$ be the neighbors of $V_i$ in $H$. If $H$ contains only the edges $(V_i,V_1),(V_i,V_2),\ldots,(V_i,V_d)$, then $k=d+1$ and $H \cong S_k$. If $k \le 3$, then $k  \le \left\lfloor \frac {k+7}3 \right\rfloor \le \left\lfloor \frac {n+7}3 \right\rfloor$. If $k=4$, then it is easy to check that $n \ge 5$ and hence $k  \le \left\lfloor \frac {n+7}3 \right\rfloor$. Suppose that $k \ge 5$, $d \ge 4$. 
By Lemma \ref{L4}, the subgraph $G[V_i]$ is disconnected. Let $K$ be an arbitrary component of $G[V_i]$ with $m$ vertices. 
Similarly to the proof of Lemma \ref{L1}, one can show that $K$ is joined with the set $V_1 \cup V_2 \cup \ldots \cup V_d$ by at most $m+2$ edges. 
Since every subgraph $G[V_i \cup V_j]$ is connected for $j=1,2,\ldots,d$, $K$ is joined by an edge with each of the sets $V_1,V_2,\ldots,V_d$. Hence $d \le m+2$ and $m \ge d-2=k-3$. As the subgraph $G[V_i]$ has at least two components, we get $n \ge 2(k-3)+d=3k-7$. Thus, $k  \le \left\lfloor \frac {n+7}3 \right\rfloor$.

Assume that $H$ has an edge $e \not \in \{(V_i,V_1),(V_i,V_2),\ldots,(V_i,V_d)\}$. Since $\alpha(H)=1$, 
the edge $e$ must be incident with all vertices $V_1,V_2,\ldots,V_d$. Therefore, $d=2$, $k=3$, and $e=(V_1,V_2)$. Thus, $H \cong K_3$. 
\end{proof}

\begin{lemma} \label{L6}
If $k \le 4$, then $H$ is isomorphic to one of the graphs $K_1$, $K_2$, $\overline{K}_2$, $P_3$, $K_3$, $\overline{K}_3$, $2K_2$, $S_4$, $P_4$, $C_4$, $C_3+e$, $C_4$, $K_4-e$, $K_4$, $\overline{K}_4$.
\end{lemma}

\begin{proof} 
If $\alpha(H) \le 1$, then by Lemmas \ref{L3} and \ref{L5}, $H$ is isomorphic to one of the graphs $\overline{K}_1 \cong K_1$, $\overline{K}_2$, $\overline{K}_3$, $\overline{K}_4$, $S_2 \cong K_2$, $S_3 \cong P_3$, $S_4$ or $K_3$. So we are left with the case $\alpha(H)=2$, $k=4$. Clearly, $H$ is isomorphic to one of the graphs $2K_2$, $P_4$, $C_4$, $C_3+e$, $C_4$, $K_4-e$ or $K_4$. 
\end{proof}

Denote by $2C_3+e$ a graph which is obtained from $K_4-e$ by
identifying its vertex of degree 3 with a vertex of $K_2$ (see Fig.~\ref{Fig3}).

\begin{figure}[ht] 
\centering
\includegraphics[width=0.2\linewidth]{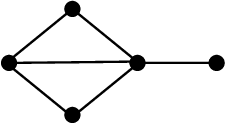}
\caption{Graph $2C_3+e$.} \label{Fig3}        
\end{figure}

\begin{lemma} \label{L7}
The coalition graph $H$ does not contain any of the following subgraphs:

(i) $K_2 \cup K_3$;

(ii) $K_2 \cup S_4$;

(iii) $2C_3+e$.
\end{lemma}

\begin{proof} 

(i) Suppose that $H$ contains a subgraph $K_2 \cup K_3$ with the vertices $V_1,V_2,\ldots,V_5$ and edges 
$(V_1,V_2)$, $(V_3,V_4)$, $(V_4,V_5)$, $(V_5,V_3)$. Then $V_3 \cup V_4$, $V_4 \cup V_5$, and $V_5 \cup V_3$ are connected dominating sets of $G$. 
This implies that every vertex in $D = V_1 \cup V_2$ is adjacent to vertices of at least two of the sets $V_3,V_4,V_5$. 
Thus, $G[D]$ is a connected subgraph with the maximum degree 1 which yields $G[D] 
\cong K_2$, $V_1=\{v_1\}$, $V_2=\{v_2\}$. Lemma \ref{L1} implies that $n \le 2|D|+2 = 6$. 
Thus, $|V_3 \cup V_4 \cup V_5| \le n-|D| \le 4$. Consider two possible cases.

{\it Case 1.} $|V_3 \cup V_4 \cup V_5|=3$, $G[V_3 \cup V_4 \cup V_5] \cong K_3$, $V_i=\{v_i\}$, $i=3,4,5$. 
Observe that every vertex $v_3,v_4,v_5$ is adjacent to at most one of the vertices $v_1$ or $v_2$. 
So either $v_1$ or $v_2$ is adjacent to at most one of the vertices $v_3,v_4,v_5$. 
Without loss of generality, assume that $v_1$ is not adjacent to $v_4$ and $v_5$.
Then $V_4 \cup V_5$ is not a dominating set in $G$, a contradiction.

{\it Case 2.} $|V_3 \cup V_4 \cup V_5|=4$, $n=6$, $V_3=\{v_3\}$, $V_4=\{v_4\}$, $V_5=\{v_5,v_6\}$, and $(v_3,v_4) \in E$. 
Since both sets $V_1 \cup V_2$ and $V_3 \cup V_4$ are dominating, the edges $(v_1,v_2)$ and $(v_3,v_4)$ do not belong to triangles in $G$. 
Without loss of generality, assume that $(v_1,v_3),(v_1,v_5),(v_2,v_4),(v_2,v_6) \in E$ and either $(v_3,v_5),(v_4,v_6) \in E$ or $(v_3,v_6),(v_4,v_5) \in E$. 
In both cases,  the connectedness of $G[V_3 \cup V_5]$ implies that $(v_5,v_6) \in E$. 
Thus, $V_5$ is a connected dominating set in $G$, a contradiction.

(ii) Suppose that $H$ contains a subgraph $K_2 \cup S_4$ with the vertices $V_1,V_2,\ldots,V_6$ and edges $(V_1,V_2)$, $(V_3,V_4)$, $(V_3,V_5)$, $(V_3,V_6)$. 
Since $V_1 \cup V_2$ is a connected dominating set in $G$, the subgraph $G[V_3 \cup V_4 \cup V_5 \cup V_6]$ has the maximum degree at most 2. 
Let $F$ be an arbitrary connected component of the subgraph $G[V_3]$. 
Then $F$ is a path or a cycle. Since the subgraphs $G[V_3 \cup V_4]$, 
$G[V_3 \cup V_5]$, and $G[V_3 \cup V_6]$ are connected, $F$ must be joined by an edge to each of the sets $V_4,V_5,V_6$. 
However, this is impossible because the maximum degree of $G[V_3 \cup V_4 \cup V_5 \cup V_6]$ does not exceed 2 and $F$ is a path or a cycle.

(iii) Suppose that $H$ contains a subgraph $2C_3+e$ with the vertices $V_1,V_2,\ldots,V_5$ and edges $(V_1,V_2)$, $(V_1,V_3)$, $(V_1,V_4)$, $(V_1,V_5)$, $(V_2,V_3)$, $(V_2,V_4)$. 
By Lemma \ref{L4},  $V_1$ is a dominating set in $G$ but the subgraph $G[V_1]$ is disconnected. 
So the maximum degree of the subgraph $G[V_2 \cup V_3 \cup V_4 \cup V_5]$ does not exceed 2. 
This means that $G[V_2 \cup V_3 \cup V_4 \cup V_5]$ (and any of its subgraphs) is a collection of paths and cycles. 

Let $F_1,F_2,\ldots,F_t$ be all connected components of the subgraph $G[V_2]$. Since $G[V_2 \cup V_3]$ 
is a connected subgraph whose maximum degree does not exceed 2, 
it follows that $G[V_2 \cup V_3]$ is a path or a cycle containing all components of $G[V_2]$ and $G[V_3]$. Hence $G[V_2 \cup V_3]$ contains a path $(F_1,P_1,F_2,P_2,\ldots,P_{t-1},F_t)$, where $P_1,P_2,\ldots,P_{t-1}$ are components of $G[V_3]$.
 If $t \ge 3$, then the vertices of $F_2$ are joined in $G[V_2 \cup V_3 \cup V_4 \cup V_5]$ by edges only to the vertices of $V_3$ (namely, to the end vertices of $P_1$ and $P_2$). Thus, $F_2$ is a connected component in the subgraph $G[V_2 \cup V_4]$. Therefore, $G[V_2 \cup V_4]$ is disconnected, a contradiction.

Assume that $t=2$. 
Then $G[V_2 \cup V_3]$ contains a path $(F_1,P_1,F_2)$. Similarly, the subgraph $G[V_2 \cup V_4]$ is connected and contains a path $(F_1,Q_1,F_2)$, where $Q_1$ is a component of $G[V_4]$. 
Therefore, the subgraph $G[V_2 \cup V_3 \cup V_4]$ is a cycle $C=(F_1,P_1,F_2,Q_1)$, 
where $P_1=G[V_3]$ and $Q_1=G[V_4]$. 
Note that $C$ is a connected component of the subgraph $G[V_2 \cup V_3 \cup V_4 \cup V_5]$. So the vertices of $V_2 \cup V_3$ are not adjacent to the vertices of $V_5$. Thus, the set $V_2 \cup V_3$ is not dominating in $G$, a contradiction.

Finally, assume that $t=1$. Then $G[V_2]=F_1$ and the subgraph $G[V_2 \cup V_3]$ contains a path $(F_1,P_1)$. Similarly, the subgraph $G[V_2 \cup V_4]$ contains a path $(F_1,Q_1)$. 
Hence the subgraph $G[V_2 \cup V_3 \cup V_4]$ is a path (or a cycle) $P=(P_1,F_1,Q_1)$,
 where $P_1=G[V_3]$ and $Q_1=G[V_4]$. 
Let $a \in P_1$ and $b \in Q_1$ be the end vertices of $P$. Observe that $V_2 \cup V_3$ 
can be joined with $V_5$ only by one edge going from $a$. 
As $V_2 \cup V_3$ is a dominating set, we get $V_5=\{v_5\}$ and $(a,v_5) \in E$. Similarly, we obtain $(b,v_5) \in E$. Thus, $v_5$ can be adjacent to only one vertex of $V_1$. 
Since the subgraph $G[V_1]$ is disconnected,  the subgraph $G[V_1 \cup V_5]$ is also disconnected, 
which is a contradiction. 
\end{proof}

\begin{lemma} \label{L8}
If $k=5$, then $H$ is isomorphic to one of the graphs $P_2\cup P_3$, $S_{1,2}$, $S_5$, $P_5$, $C_3+e+e$, $C_3+2e$, $C_4+e$, $C_5$, $K_{2,3}$.
\end{lemma}

\begin{proof} 
By Lemma \ref{L3}, $H$ has no isolated vertex. If $H$ is a forest, then $H$ is isomorphic to one of the graphs $P_2\cup P_3$, $S_{1,2}$, $S_5$ or $P_5$. Suppose that $H$ is not a forest. Let $C=(V_1,V_2,\ldots,V_m)$ be the shortest cycle in $H$. 
Then $C$ has no chords. Clearly, if $m=k=5$, then $H=C \cong C_5$. Suppose that $m=4$. 
Since $V_5$ is not an isolated vertex, we can assume that $V_5$ is adjacent to $V_1$. 
If $V_5$ is not adjacent to other vertices of $H$, then $H \cong C_4+e$. If $V_5$ is adjacent to $V_2$ or $V_4$, then $H$ contains a subgraph $K_2 \cup K_3$, which contradicts Lemma \ref{L7}(i). 
If $V_5$ is adjacent to $V_3$, then $H \cong K_{2,3}$.

Finally, assume that $m=3$. By Lemma \ref{L7}(i), we have  $(V_4,V_5) \not\in E(H)$. So each vertex $V_4, V_5$ is adjacent to a vertex of $C$. If $V_4$ or $V_5$ is adjacent to at least two vertices of $C$, then $H$ contains one of the subgraphs $K_2 \cup K_3$ or $2C_3+e$, which contradicts Lemma \ref{L7}(i),(iii). Thus, each vertex $V_4,V_5$ is adjacent to exactly one vertex of $C$. If they are adjacent to the same vertex of $C$, then $H \cong C_3+2e$, otherwise, $H \cong C_3+e+e$. 
\end{proof}

\begin{lemma} \label{L9}
If $k \ge 6$, then $H$ is isomorphic to one of the graphs $M_6$, $3K_2$, $S_{2,2}$, $S_k$.
\end{lemma}

\begin{proof} 
By Lemma \ref{L3}, $H$ has no isolated vertex. 
If $\alpha(H)=1$, then by Lemma \ref{L5}, we have $H \cong S_k$. If $\alpha(H)=3$, then by Lemma \ref{L2}, $H$ is isomorphic to one of the graphs $M_6$ or $3K_2$. So we are left with the case $\alpha(H)=2$. 
Let $(V_1,V_2)$ and $(V_3,V_4)$ be edges of $H$. Corollary \ref{Corol1}, applied to the sets $D_1 = V_1 \cup V_2$ and $D_2 = V_3 \cup V_4$, gives $|D_1|+|D_2|=|V_1|+|V_2|+|V_3|+|V_4| \ge n-2$. Hence $G$ contains at most two vertices $v_5,v_6$ outside $D_1 \cup D_2$. 
Since $k\ge 6$, we get $|D_1|+|D_2|=n-2$, $k=6$, $V_5=\{v_5\}$, $V_6=\{v_6\}$. By Corollary \ref{Corol1}, both subgraphs $G[D_1]$ and $G[D_2]$ are paths with $n/2-1$ vertices and $D_1$ is joined with $D_2$ by a matching of size $n/2-1$ in $G$. 

The condition $\alpha(H)=2$ implies that $(V_5,V_6) \not\in E(H)$. Thus, each vertex $V_5,V_6$ is adjacent to one of the vertices $V_1,V_2,V_3,V_4$ in $H$. 
Assume that $V_5$ is adjacent to $V_1$. If $V_6$ is adjacent to $V_2$ or $V_1$, then either $\alpha(H)=3$ or $H$ contains the subgraph $K_2 \cup S_4$, which contradicts Lemma \ref{L7}(ii). So we can assume that $V_6$ is adjacent to $V_3$. Since $\alpha(H)=2$ and $H$ does not contain $K_2 \cup K_3$ and $K_2 \cup S_4$, the vertices $V_2,V_4,V_5,V_6$ are not incident with the edges other than $(V_2,V_1)$, $(V_5,V_1)$, $(V_4,V_3)$, $(V_6,V_3)$ in $H$. So $H$ can have only one edge $(V_1,V_3)$ other than $(V_2,V_1)$, $(V_5,V_1)$, $(V_4,V_3)$, $(V_6,V_3)$.

We now prove that $(V_1,V_3) \in E(H)$. Corollary \ref{Corol1}, applied to the sets $D_3 = V_1 \cup V_5$ and $D_4 = V_3 \cup V_6$, shows that $V_2=\{v_2\}$, $V_4=\{v_4\}$, and both subgraphs $G[D_3]$ and $G[D_4]$ are paths with $n/2-1$ vertices. By Lemma \ref{L1}, $v_2$ is adjacent to only one vertex of the set $D_3 = V_1 \cup V_5$ and hence $v_2$ is adjacent to only one vertex of $V_1$. Therefore, $v_2$ is an end vertex of the path $G[D_1]=G[V_1 \cup V_2]$. Similarly, $v_5$ is an end vertex of the path $G[D_3]=G[V_1 \cup V_5]$. Since $D_2 = V_3 \cup V_4$ is a dominating set of $G$, the subgraph $G[V_1 \cup V_2 \cup V_5]$ has the maximum degree at most 2. This implies that $G[V_1 \cup V_2 \cup V_5]$ is a path $(v_2,P_1,v_5)$, where $P_1=G[V_1]$. Similarly, the subgraph $G[V_3 \cup V_4 \cup V_6]$ is a path $(v_4,P_3,v_6)$, where $P_3=G[V_3]$. Thus, $V_1 \cup V_3$ is a dominating set in $G$ and the subgraph $G[V_1 \cup V_3]$ has at most two components. Consider an arbitrary vertex $v \in V_1$. Note that $v$ has degree 2 in $G[V_1 \cup V_2 \cup V_5]$. As $V_3 \cup V_4$ and $V_3 \cup V_6$ are dominating sets of $G$, $v$ is either adjacent to a vertex of $V_3$ or to both vertices $v_4$ and $v_6$. In the latter case, $v$ has degree at least 4 in $G$, which is a contradiction. So $v$ is adjacent to a vertex of $V_3$. Hence the subgraph $G[V_1 \cup V_3]$ is connected,
 and $V_1 \cup V_3$ is a connected dominating set in $G$. Therefore, $(V_1,V_3) \in E(H)$ and $H \cong S_{2,2}$. 
\end{proof}

\section{Coalition graphs defined by a finite number of subcubic graphs}

To prove Theorem~\ref{Th2}, we demonstrate that a part of coalition graphs  
are defined by a finite number of subcubic graphs
while the other part  of coalition graphs are defined by an infinite number of 
 M\"{o}bius ladders and their modifications.

\begin{proposition} \label{Pr1}
Let $H$ be one of the graphs 
 $K_1$, $\overline{K}_2$,  $\overline{K}_3$, $K_4$, $\overline{K}_4$, $C_5$, $3K_2$, $K_{2,3}$ or $K_{3,3}$. 
Then $H$ is a coalition graph for at least one subcubic graph $G$, and any such a graph $G$ has order $n \le 10$. 
Hence $H$ is defined only by a finite number of subcubic graphs.
\end{proposition} 

\begin{proof}
 If $H$ is isomorphic to one of the graphs $K_1$, $\overline{K}_2$,  $\overline{K}_3$ or $\overline{K}_4$, then 
 $n \le 4$  and $G$ is isomorphic to $K_1$, $K_2$, $K_3$ or $K_4$, respectively,  by Lemma \ref{L3}. 
  The realizations of coalition graphs $K_4$, $C_5$, $3K_2$, $K_{2,3}$, and $K_{3,3}$ are shown 
 in Fig.~\ref{Fig4}. Let us prove the upper bound $n \le 10$.
  If $H$ is isomorphic to $3K_2$ or $K_{3,3} \cong M_6$, then $n=6$ by Lemma \ref{L2}. 
 Assume that $H=(V_1,V_2,\ldots,V_5) \cong C_5$. Corollary \ref{Corol1}, applied to the sets $D_1 = V_2 \cup V_3$ and $D_2 = V_4 \cup V_5$, implies that $|V_1| \le 2$. Similarly, $|V_i| \le 2$ for $i=2,3,4,5$. Hence $n \le 10$.

Assume that $H \cong K_4$. Let $m_i$ be the number of edges in the subgraph $G[V_i]$ for $i=1,2,3,4$ and  $m=m_1+m_2+m_3+m_4$. For every pair of indices $i,j \in \{1,2,3,4\}$, denote by $m_{ij}$ the number of edges joining $V_i$ with $V_j$ in $G$. Then the number of edges of the subgraph $G[V_i \cup V_j]$ is equal to $m_i+m_j+m_{ij}$, and the number of edges of $G$ is equal to $|E|=m + \sum_{i,j} m_{ij}$. Since $V_2 \cup V_3$, $V_2 \cup V_4$, and $V_3 \cup V_4$ are dominating sets in $G$, every vertex of $V_1$ is adjacent to vertices of at least two sets $V_2,V_3,V_4$. So the maximum vertex degree of the subgraph $G[V_1]$ does not exceed 1, and the edge set of $G[V_1]$ is a matching. Hence $m_1 \le |V_1|/2$. Similarly, $m_i \le |V_i|/2$ for $i=2,3,4$. Therefore, $m \le n/2$. Since the subgraph $G[V_i \cup V_j]$ is connected, we have $m_i+m_j+m_{ij} \ge |V_i|+|V_j|-1$. Summarizing these inequalities over all pairs of indices $i,j$ gives 
$$
3m+\sum_{i,j}m_{ij} = 2m + |E| \ge 3(|V_1| + |V_2| + |V_3| + |V_4|)-6=3n-6.
$$ 
As $G$ is subcubic, we have $|E| \le 3n/2$. Thus, $2m \ge 3n-6-|E| \ge 3n/2-6$. 
Combining this inequality with $m \le n/2$ implies that $n \le 12$, $n \neq 11$, and if $n=12$, then $m=n/2=6$. So if $n=12$, then the edge set of every subgraph $G[V_i]$ is a perfect matching for $i=1,\ldots,4$. 
Thus, all the numbers $|V_1|,\ldots,|V_4|$ are even. Since $|V_1|+\ldots+|V_4|=12$, at least two of the numbers $|V_1|,\ldots,|V_4|$ are equal to 2. 
Let $|V_1|=|V_2|=2$. By Lemma \ref{L1}  for the set $D = V_1 \cup V_2$, it follows that $|D|=|V_1|+|V_2| \ge 12/2-1=5$, which is a contradiction. Thus, $n \le 10$.

Finally, assume that $H \cong K_{2,3}$ and $H$ has edges $(V_1,V_3)$, $(V_1,V_4)$, $(V_1,V_5)$, $(V_2,V_3)$, $(V_2,V_4)$, $(V_2,V_5)$. Corollary \ref{Corol1}, applied to the sets $D_1 = V_1 \cup V_4$ and $D_2 = V_2 \cup V_5$, implies that $|V_3| \le 2$. Similarly, $|V_i| \le 2$ for $i=4,5$. Let $F$ be an arbitrary connected component of the subgraph $G[V_1]$. 
Since $V_2 \cup V_3$, $V_2 \cup V_4$, and $V_2 \cup V_5$ are dominating sets in $G$, every vertex of $F$ is either adjacent to a vertex of $V_2$ or is adjacent to vertices of all three sets $V_3$, $V_4$, and $V_5$. In the latter case, $F$ is a single vertex. In the former case, $F$ is a path or a cycle, which is joined by at most two edges with the set $V_3 \cup V_4 \cup V_5$. 
Therefore, one of the subgraphs $G[V_1 \cup V_3]$, $G[V_1 \cup V_4]$ or $G[V_1 \cup V_5]$ is disconnected, a contradiction. So every component of $G[V_1]$ is a single vertex $v$, which is adjacent to vertices of $V_3$, $V_4$, and $V_5$. 
Thus, $v$ has degree 1 in the subgraph $G[V_1 \cup V_3]$. Corollary \ref{Corol1}, applied to the sets $D_3 = V_1 \cup V_3$ and $D_4 = V_2 \cup V_4$, implies that $G[V_1 \cup V_3]$ is a path or a cycle. Hence $G[V_1 \cup V_3]$ has at most two vertices of degree 1 and so $|V_1| \le 2$. 
Similarly, we have $|V_2| \le 2$. Thus, $n=|V_1|+|V_2|+ \ldots +|V_5| \le 5 \cdot 2 = 10$.

The upper bound of Proposition~\ref{Pr1} is sharp because the coalition graph $C_5$ is defined by the ladder $M_{10}$ 
and the coalition graph $K_{2,3}$ is defined by a cubic graph of order 10 (see Fig.~\ref{Fig4}).
\end{proof}

\newpage

\begin{figure}[ht] 
\centering
\includegraphics[width=0.8\linewidth]{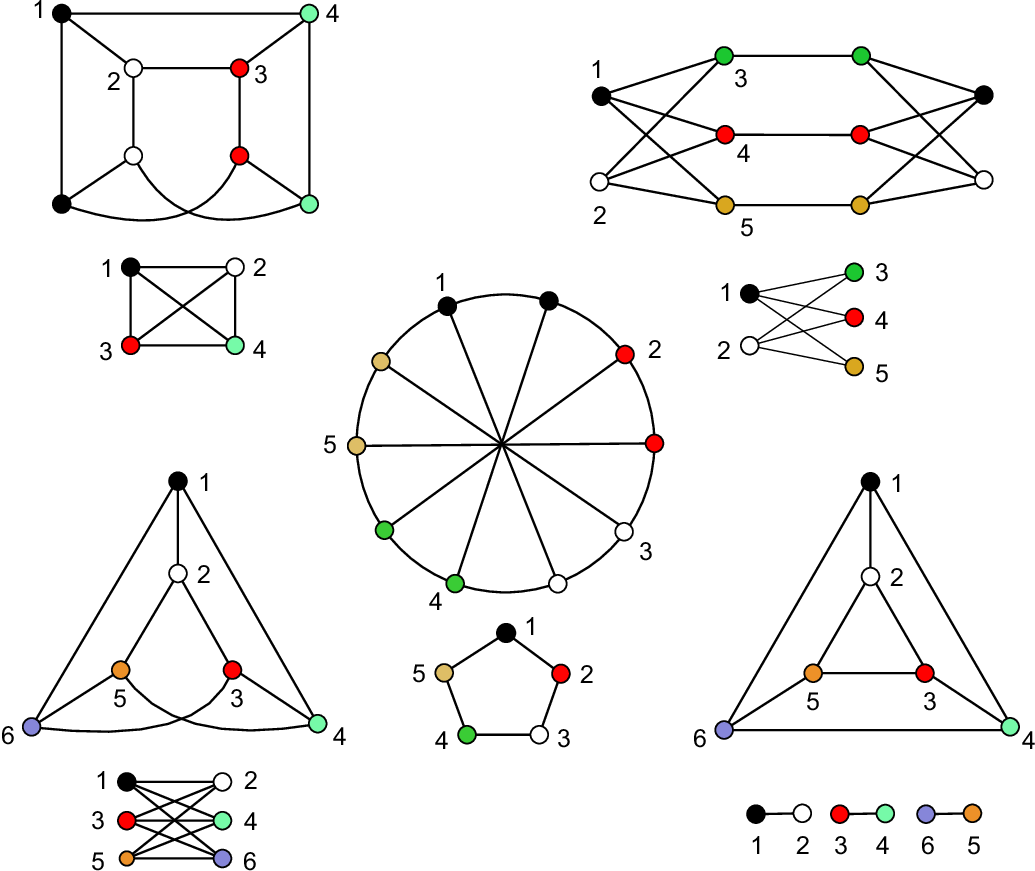}
\caption{Coalition graphs $K_4$, $K_{2,3}$, $C_5$,  $K_{3,3}$, and $3K_2$.} \label{Fig4}
\end{figure}

\section{Coalition graphs defined by an infinite number of subcubic graphs}

In this section, we determine coalition graphs that are generated by an infinite number of 
 M\"{o}bius ladders and their modifications.

\begin{proposition} \label{Pr2}
Coalition graphs 
$C_3$, $2K_2$, $P_4$, $C_4$, $C_3+e$, $K_4-e$, $P_2\cup P_3$, $S_{1,2}$, $P_5$,
$C_3+e+e$, $C_3+2e$, $C_4+e$, $S_{2,2}$, and $S_k$, $k\ge 2$, are defined by an infinite number of subcubic graphs.
\end{proposition} 

A graphical proof of this statement is presented in Figs.~\ref{Fig5},~\ref{Fig6}. 
Stars $S_k$  are ge\-ne\-ra\-ted by the subcubic graph $F_1$, where $F_1[V_1]$ and $F_1[V_2]$ 
are arbitrary connected subgraphs.
The corresponding partition of $F_1$ is $\pi =\{V_1\cup V_2,\{ v_1\},\{ v_2\},\ldots ,\{ v_{k-1}\}\}$, 
and the set $ V_1 \cup V_2 \cup  \{v_i \}$ forms a connected coalition for $i=1,2,\ldots,k-1$.  
All other coalition graphs of Proposition~\ref{Pr2}, except $C_4+e$, can be defined by infinite families of M\"{o}bius ladders of order $n \ge 10$. 
The corresponding coalition graphs are shown in  Figs.~\ref{Fig5},~\ref{Fig6} below the ladders.
Vertices of each set of $\pi$ have the same color and are marked by the number of the corresponding coalition set. 
The sets $V_1$ and  $V_2$ of $\pi$ consist of black and white vertices, respectively (only one vertex of these sets 
is marked).
To obtain the coalition graph $C_4+e$, the structure of a M\"{o}bius ladder 
has been modified (see graph $F_2$ in Fig.~\ref{Fig6}).  

Observe that Lemmas~\ref{L6},~\ref{L8},~\ref{L9} and Propositions~\ref{Pr1},~\ref{Pr2} imply Theorems~\ref{Th1} and \ref{Th2}.

\begin{figure}[pht] 
\centering
\includegraphics[width=\linewidth]{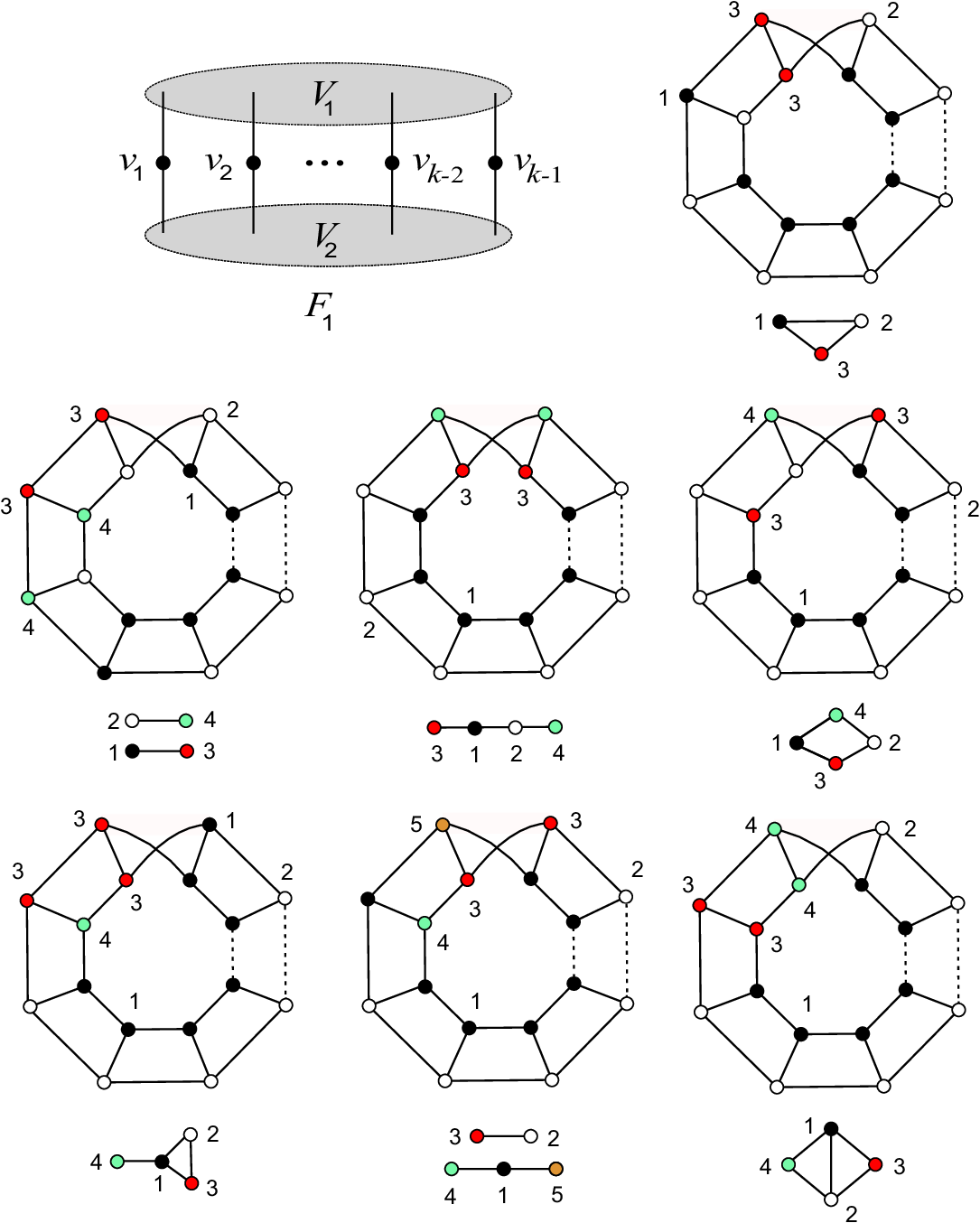}
\caption{Coalition graphs for  M\"{o}bius ladder $M_n$ of order $n\ge 10$.} \label{Fig5}
\end{figure}

\begin{figure}[pht] 
\centering
\includegraphics[width=\linewidth]{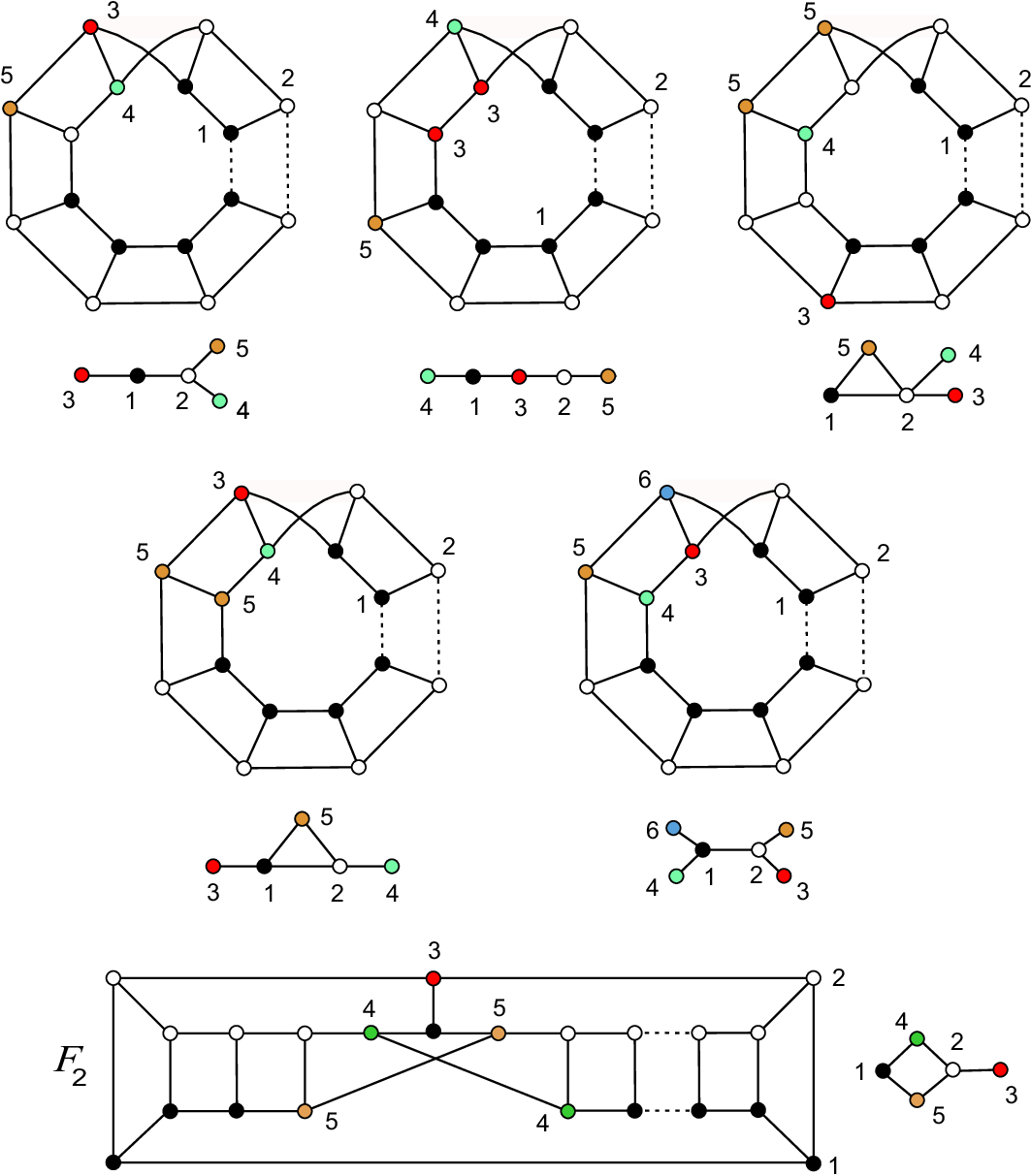}
\caption{Coalition graphs for  M\"{o}bius ladder $M_n$ and graph $F_2$ of order $n\ge 10$.} \label{Fig6}
\end{figure}


\section{Coalition graphs for prisms}

Since M\"{o}bius ladders can be obtained from prisms only by switching two edges,
it is interesting to compare their sets of coalition graphs. Table~\ref{Tab2} 
contains the number of coalition graphs for prisms $Pr_n$ of order $n \le 18$.
These prisms define only 14 coalition graphs only one of which is not generated by M\"{o}bius ladders.
Namely, Lemma~\ref{L2} implies that $3K_2$ cannot be a coalition graph for  $M_n$.
Note that the smallest prism $Pr_6$ defines significantly more coalition graphs than the ladder $M_6$.

\begin{table}[h!]
\centering
\caption{Number of coalition graphs $CCG$ for prism $Pr_n$, $6 \le n \le 18$.} \label{Tab2}
\begin{tabular}{rlrrrrrrr}  \hline
$ N$ &  $CCG$ &  $Pr_6$ & $Pr_8$ & $Pr_{10}$ & $Pr_{12}$ & $Pr_{14}$ & $Pr_{16}$ & $Pr_{18}$  \\  \hline
 1. & $K_2$ &          3 &         28 &        190 &       1112 &       5572 &      25622 &     111906      \\  \hline  
 2. & $P_3$ &         18 &        312 &       4140 &      42988 &     363636 &    2716026 &   18741699      \\ 
 3. & $C_3$ &          8 &        216 &       1870 &      10826 &      54201 &     248748 &    1074777      \\  \hline  
 4. & $S_4$ &          6 &        140 &       2660 &      24268 &     177562 &    1169328 &    7240476      \\ 
 5. & $2K_2$ &          3 &         63 &        775 &       9913 &      84875 &     586045 &    3542679     \\ 
  6. & $P_4$ &         18 &        264 &       2350 &      14208 &      68726 &     286084 &    1079694      \\ 
 7. & $C_3+e$ &          . &        240 &       1180 &       4740 &      15260 &      44032 &     119196      \\ 
 8. & $C_4$ &          3 &         15 &         45 &        219 &        560 &       1140 &       2043      \\ 
 9. & $K_4-e$ &          6 &         48 &        165 &        246 &        343 &        456 &        585      \\ 
10. & $K_4$ &          . &         14 &          . &          . &          . &          . &          .      \\ \hline 
11. & $S_5$ &          . &          6 &        130 &       1362 &      12068 &      97648 &     746784      \\ 
12. & $P_2 \cup P_3$ &          6 &          . &          . &          . &          . &          . &          .     \\ 
13. & $S_{1,2}$ &          6 &          . &          . &          . &          . &          . &          .      \\   \hline
 14. & $3K_2$ &          1 &          . &          . &          . &          . &          . &          .    \\  \hline
\end{tabular}
\end{table}

\section{Connected coalition number of subcubic graphs}

Trees and cycles are the simplest classes of subcubic graphs.
It was shown that  $CC(T ) = 2$ for any tree $T$ of order $n \ge 3$ without a full vertex \cite{Alik2023}.
For any cycle of order $n$, $CC(C_4)=4$ and $CC(C_n)=3$ if $n \not =  4$  \cite{Guan2024}.

It is known that the coalition number $C(G)$, based on dominating sets in graphs,  is bounded by the maximum vertex degree of $G$. 
Namely,  $C(G)\leq\ (\Delta(G)+3)^2/4$  for any graph $G$ \cite{Hayn2021}.
For cubic graphs, this bound yields $C(G) \leq 9$.
It was proved that the coalition number of cubic graphs of order $n \le 10$ is at most 8, 
and the smallest cubic graphs with $C(G) = 9$ have 16 vertices \cite{Alik2023a,Dobr2024}. 
Table~\ref{Tab3} represents the number of cubic graphs of order $n$ having the connected 
coalition number $k$ for each even $n=4,6,\ldots,14$ and for each $k=2,3,\ldots,7$.

\begin{table}[h!] 
\centering
\caption{Number of cubic graphs of order $n$ \\ with  connected coalition number $k$.} \label{Tab3}
\begin{tabular}{c|rrrrrr}  \hline
 \ \ \ \  $k\, \backslash \, n$ & $4$  & 6 &   8  &  10   &  12  &   14  \\  \hline
 2 & .  & . &   .  &   .   &  .   &     .    \\ 
 3 & .  & . &   .  &   1   &  .   &  1      \\ 
 4 & 1  & . &   .  &   1   &  9   &  36    \\  
 5 & .  & . &   4  &   15  &  32  &  67     \\  
 6 & .  & 2 &   1  &   2   &  44  &  268    \\ 
 7 & .  & . &   .  &   .   &   .  &  137   \\  \hline
{\rm total}   & 1  & 2 &   5  &   19  &  85  &  509   \\  
\end{tabular}
\end{table}


Based on the results of the previous sections, we prove  
the following upper bound for the connected coalition number of subcubic graphs. 

\newpage

\begin{theorem} \label{Th3}
For the connected coalition number of a subcubic graph $G$,
the following inequality holds
$$
CC(G) \le \max \left\{ 6,\left\lfloor \frac {n+7}3 \right\rfloor \right\}
$$ 
where the bound is sharp for $n \in \{ 6,8 \}$ and for $n\ge 10$. If $CC(G)=k>6$, 
then the only coalition graph of $G$ of order $k$ is $S_k$.
\end{theorem}

\begin{proof}
 If the coalition graph $H \cong S_k$, then by Lemma \ref{L5}, we have $k \le \left\lfloor \frac {n+7}3 \right\rfloor$. 
 Lemma \ref{L9} implies that if $H \not\cong S_k$, then $k \le 6$, and if $CC(G)=k>6$, 
 then the only coalition graph of $G$ of order $k$ is $S_k$. If $n \in \{ 6,8,10 \}$, 
 then $\max \left\{ 6,\left\lfloor \frac {n+7}3 \right\rfloor \right\} = 6$ 
 and the upper bound is sharp because the coalition graph $3K_2$ is defined by the prism $Pr_6$ while the coalition graph 
 $S_{2,2}$ is defined by the ladders $M_8$ and $M_{10}$ 
 (see Fig.~\ref{Fig4}). If $n \ge 11$, then $\max \left\{ 6,\left\lfloor \frac {n+7}3 \right\rfloor \right\} = \left\lfloor \frac {n+7}3 \right\rfloor$ and the bound is sharp because the coalition graph $S_k$ for $k \ge 6$ can be defined by subcubic graphs of order 
 $3k-7$, $3k-6$, and $3k-5$ (see Fig.~\ref{Fig7}). 
\end{proof}

\begin{figure}[th] 
\centering
\includegraphics[width=0.9\linewidth]{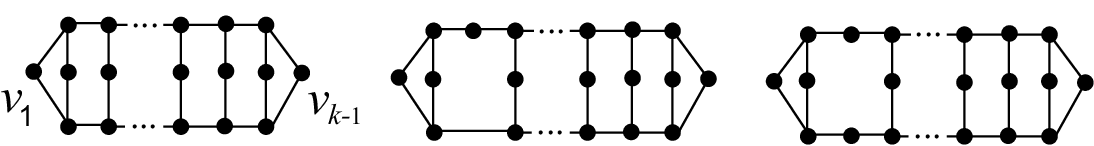}
\caption{The smallest subcubic graphs for which the exact upper bound  is achieved.} \label{Fig7}
\end{figure}

\section{Conclusion}
In this paper, all coalition graphs of connected coalition partitions 
of subcubic graphs are described. 
The upper bound on the coalition number of subcubic graphs is established
in terms of the graph order.

A graph of  some family is called universal if it  defines all coalition graphs 
that are generated by all graphs of the family. For families of all paths or all cycles,
it is known that there is no universal graph in the case of standard domination
\cite{Hayn2023a}.
Therefore, the following  problem  arises: 
find a  graph defining the maximal number of coalition graphs (mc-graph) for a given family. 
If such an mc-graph exists, then what is the minimal order of the  graph?
It is known that each path $P_n$ of order $n\ge 10$ is an mc-graph for the family of all paths,
and each cycle $C_{3k}$ for $k\ge 5$ is an mc-graph for the family of all cycles 
\cite{Dobr2024-2,Gleb2025,Hayn2023a}.
Our considerations imply that there is no universal graph
for connected domination in subcubic graphs. 
Therefore, it would be interesting to find a subcubic mc-graph.

\section{Acknowledgement}
This work was supported by the state contract of the Sobolev Institute of Mathematics (project number FWNF-2022-0017).

\end{document}